\documentclass[11pt, a4paper]{article}
\usepackage{amsfonts}
\usepackage{}
\usepackage{amsthm}
\usepackage{amsmath,amssymb,latexsym,color}
\usepackage[mathscr]{eucal}
\usepackage[colorlinks,
            linkcolor=blue,
            anchorcolor=green,
            citecolor=magenta
           ]{hyperref}
\usepackage{graphicx}
\usepackage{tikz}
\usetikzlibrary{calc}
\usepackage{authblk}

\long\def\delete#1{}

\usepackage{color}

\definecolor{Blue}{rgb}{0,0,1}
\definecolor{Red}{rgb}{1,0,0}
\definecolor{DarkGreen}{rgb}{0,0.6,0}
\definecolor{DarkYellow}{rgb}{1,1,0.2}
\definecolor{DarkPurple}{rgb}{.6,0,1}

\usepackage{xcolor}
\usepackage[normalem]{ulem}

\usepackage{cleveref}
\crefformat{section}{\S#2#1#3}
\crefformat{subsection}{\S#2#1#3}
\crefformat{subsubsection}{\S#2#1#3}
\crefrangeformat{section}{\S\S#3#1#4 to~#5#2#6}
\crefmultiformat{section}{\S\S#2#1#3}{ and~#2#1#3}{, #2#1#3}{ and~#2#1#3}

\def\q{\hfill\rule{1ex}{1ex}}
 \newcommand{\BF}{{\mathbb {F}}}

 \newcommand{\BZ}{{\mathbb {Z}}}

\newcommand{\qb}[2]{{#1 \brack #2}_q}
\newcommand{\qp}[1]{\langle #1 \rangle _q}

\newcommand{\spann}[1]{\left\langle #1 \right\rangle}
\newcommand{\ls}{\leqslant}
\newcommand{\gs}{\geqslant}

\begin{document}
\setcounter{page}{1}
\newtheorem{thm}{Theorem}[section]
\newtheorem{fthm}[thm]{Fundamental Theorem}
\newtheorem{dfn}[thm]{Definition}
\newtheorem*{rem}{Remark}
\newtheorem{lem}[thm]{Lemma}
\newtheorem{cor}[thm]{Corollary}
\newtheorem{exa}[thm]{Example}
\newtheorem{prop}[thm]{Proposition}
\newtheorem{prob}[thm]{Problem}
\newtheorem{fact}[section]{Fact}
\newtheorem{con}[thm]{Conjecture}
\renewcommand{\thefootnote}{}
\newcommand{\remark}{\vspace{2ex}\noindent{\bf Remark.\quad}}
\newtheorem{ob}[thm]{Observation}
\newcommand{\rmnum}[1]{\romannumeral #1}
\renewcommand{\abovewithdelims}[2]{%
\genfrac{[}{]}{0pt}{}{#1}{#2}}

\newcommand\Sy{\mathrm{S}}
\newcommand\Cay{\mathrm{Cay}}
\newcommand\tw{\mathrm{tw}}
\newcommand\supp{\mathrm{supp}}

%-------------------  First Head  -----------------------------------------

\def\qed{\hfill$\Box$\vspace{11pt}}

\title {\bf  On the $ P_3 $-hull numbers of $ q $-Kneser graphs and Grassmann graphs}

\author{Jiaqi Liao\thanks{ E-mail: \texttt{liaojq19@mails.tsinghua.edu.cn}}}
\author{Mengyu Cao\thanks{Corresponding author. E-mail: \texttt{caomengyu@mail.bnu.edu.cn}}}

\author{Mei Lu\thanks{E-mail: \texttt{lumei@tsinghua.edu.cn}}}

\affil{\small Department of Mathematical Sciences, Tsinghua University, Beijing 100084, China}

\date{}

\openup 0.5\jot
\maketitle

\begin{abstract}
Let $S$ be an $n$-dimensional vector space over the finite field $\mathbb{F}_q$, where $q$ is necessarily a prime power. Denote  $K_q(n,k)$ (resp. $J_q(n,k)$) to be the \emph{$q$-Kneser graph} (resp. \emph{Grassmann graph}) for $k\geq 1$ whose vertices are the $k$-dimensional subspaces of $S$ and two vertices $v_1$ and $v_2$ are adjacent if $\dim(v_1\cap v_2)=0$ (resp. $\dim(v_1\cap v_2)=k-1$). 
We consider the infection spreading in the $ q $-Kneser graphs and the Grassmann graphs: a vertex gets infected if it has at least two infected neighbors. In this paper, we compute the $ P_3 $-hull numbers of $K_q(n,k)$ and $J_q(n,k)$ respectively, which is the minimum size of a vertex set that eventually infects the whole graph.

\vspace{2mm}

\noindent{\bf Key words}\ \ $ P_3 $-hull number, $ q $-Kneser graph, Grassmann graph

\

\noindent{\bf MSC2010:} \   05C36, 52A37, 05A30

\end{abstract}

\section{Introduction}

Throughout this paper graphs are finite and undirected with no loops or multiple edges. Let $G=(V,E)$ be a graph with vertex set $V$ and edge set $E$. For any vertex $u\in V(G)$, the  neighborhood of $u$ in $G$, denoted by $N(u)$, is the set of vertices adjacent to $u$. For a graph $ G = (V, E) $ and a subset $ T \subseteq V $, the \emph{$ P_3 $-interval} $ I[T] $ is the union of $ T $ and the vertices with at least two neighbors in $ T $. $ T $ is said to be \emph{$ P_3 $-convex} if $ I[T] = T $. For $ p \in \BZ_{\geqslant 0} $, set $ I^0[T] := T $ and $ I^{p + 1}[T] := I[I^p[T]] $. Then $ H(T) := I^\infty[T] $ is the \emph{$ P_3 $-convex hull} of $ T $, and it is the smallest $ P_3 $-convex set containing $ T $. $ T $ is said to be a \emph{$ P_3 $-hull set} of $ G $ if $ H(T) = V $, and the \emph{$ P_3 $-hull number} $ h(G) $ of $G$ is the cardinality of a minimum $ P_3 $-hull set of $G$. From the definition, $h(G)\ge 2$.

$P_3$-convexity was introduced with the aim of modeling the spread of a disease in a grid~\cite{Bollobas}. Since then, there are many articles investigate this convexity, we refer the readers for instance to \cite{Bollobas,Bresar,Centeno}. Related to complexity aspects, Centeno et al. \cite{Centeno} proved that, given a graph $G$ and an integer $k$, to decide whether the $P_3$-hull number of $G$ is at most $k$ is an NP-complete problem.  Coelho et al. \cite{Cohelo} showed that compute the $P_3$-hull number is NP-complete even in bipartite graphs. Polynomial-time algorithms to compute the hull number of some special classes of graphs were given \cite{AMSSW, DG}.
Recently, Grippo et al. \cite{GPTVV2021} studied the $P_3$-hull number of Kneser graphs $K(n,k)$ for $ n \geqslant 2k + 2 $, whose vertices are the $k$-subsets of a $n$-set $X$, and two vertices $v_1$ and $v_2$ are adjacent if $|v_1\cap v_2|=0$. Bres\u{a} et al. \cite{Bresar} dealt with the problem of computing the $P_3$-hull number of Hamming graphs.

In this paper we  deal with the problem of computing the $ P_3 $-hull number of $ q $-Kneser graphs and Grassmann graphs, respectively. $ q $-Kneser graphs are the $q$-analogues of Kneser graphs~\cite{Kneser,Lovasz}. For a fixed prime power $ q $, let $ n, k \in \BZ $ with $ n \geqslant 2k \geqslant 2 $, and $ S $ an $ n $-dimensional vector space over the finite field $ \BF_q $. Denote by $ \qb{S}{k} $ the family of all $ k $-dimensional subspaces of $ S $. In the sequel, for two vector spaces $ w_i$ and $w_j$, we write $ w_{ij} $ for $ w_i \cap w_j $. The \emph{$ q $-Kneser graph} $ K_q(n, k) $ is the graph whose vertex set is $ \qb{S}{k} $ and two vertices $ w_1 $ and $ w_2 $ are adjacent if $ w_{12} = \textbf{0} $, the zero space. The {\em Grassmann graph} $J_q(n,k)$ is the graph whose vertex set is $ \qb{S}{k} $ and two vertices $ w_1 $ and $ w_2 $ are adjacent if $\dim w_{12} = k-1 $. Over the years, several aspects of $q$-Kneser graphs such as treewidth, chromatic number, energy, eigenvalues and some other properties had been widely studied as one can find in, for example \cite{Blokhuis,CLLL2021,Huang,Lv,Lv2}.  The Grassmann graph plays an important role in geometry \cite{Pan,Wan} and coding theory \cite{KK,PWL}.

Now we fix some notations. The \emph{$ q $-limit} and the \emph{Gaussian binomial coefficient} of $ n, k \in \BZ_{> 0} $ is defined respectively by

\[\qp{n} := \prod_{i = 1}^{n}(q^i - 1), \qquad \qb{n}{k} := \frac{\qp{n}}{\qp{k}\qp{n - k}}.\]

\noindent In addition, we set $\qp{0}=1$, $\qb{n}{0}=\qb{0}{0}=1$ and $\qb{n}{k}=0$ if $k>n$. Note that the Gaussian binomial coefficient $\qb{n}{k}$ is the cardinality of $ \qb{S}{k} $.

 The main results in this paper are as follows.

\begin{thm}\label{thm_main}
	
When $ n \geqslant 2k\ge 2 $, we have $ h(K_q(n, k)) = 2 $.	
	
\end{thm}

\begin{thm}\label{thm_main2}
	
For $ n \gs 2k\ge 2 $, we have $ h(J_q(n, k)) = 2 $.	
	
\end{thm}

Note that $K_q(n, 1)=J_q(n,1)$ is a complete graph. Then $h(K_q(n, 1))=  h(J_q(n, 1)) = 2$. Hence we will assume $k\ge 2$ in the following sections. This article is organized as follows. In Section 2, we will give the exact value of $h(K_q(n, k))$. The exact value of $ h(J_q(n, k))$ will be given in Section 3.

\section{Hull number of $K_q(n, k)$}\label{firstcase}

In the sequel, we will abbreviate ``$ k $-dimensional subspace" to ``$ k $-subspace". Recall $ S $ is an $ n $-dimensional vector space over $ \BF_q $. Given an $ m $-subspace $ u(m)\subseteq S $, define
\[A(S,u(m),k,i) := \left\{w\subseteq S: \dim w = k, \dim (w \cap u(m)) = i\right\},\]where $0\le i\le k$. Let  $a(S,u(m),k,i)=|A(S,u(m),k,i)|$.
 Note that $A(S,u_m,k,0),A(S,u_m,k,1),$ $\ldots, A(S,u_m,k,k)$ form a partition of $V(K_q(n, k))$.

\begin{lem}\label{S_i}
	
{\rm (Lemma 9.3.2 in \cite{GK2015})} Let $ u_m\subseteq S $ be an $ m $-subspace. Then for $0\le i\le k$, we have	
\[a(S,u(m),k,i) = q^{(k-i)(m-i)}\qb{m}{i}\qb{n-m}{k-i}.\]
\end{lem}

By Lemma \ref{S_i}, $ K_q(n, k) $ is $ q^{k^2}\qb{n-k}{k} $-regular and the size is $ \frac{1}{2}q^{k^2}\qb{n - k}{k}\qb{n}{k} $.
For $0\le i,j\le k$ with $i\not=j$ and $0\le m\le n-k$, define
\[d_{ij} = q^{ij}\qb{m - i}{j}\times\sum_{r = b}^{B} q^{r({r - 1} + 2i)/2} \qb{m - i - j}{r}\qb{k - i}{r}\qp{r}\times q^{s(s + m - i)}\qb{n - m - k + i}{s},\]
where $ s = k - j - r $, $ b = \max\{0, m +2k - n - i - j\}  $, $ B = \min\{m - i - j, k - i, k - j\} $. %If $ b > B $, then we set $ d_{ij} = 0 $.

Note that  $d_{ij}\not=d_{ji}$ and $ d_{ij} = 0 $ if and only if $b>B$. We have the following result first.

\begin{lem}\label{d_ij}
Let $n\ge 2k+1$ and $m=k+1$. Then for $1\le i\le k$, we have $d_{i0}\ge 2$.
\end{lem}
%When $ n \geqslant 2k + 1,~ m = k + 1 $, we have $ d_{ij} \ne 1 $.
\begin{proof} Since $m=k+1$, we have $B=k-i$  and $ b = \max\{0, 3k + 1 - n - i\} $. If $d_{i0}=0$, then $k-i<\max\{0, 3k + 1 - n - i\} $ which implies $k<\max\{i, 3k + 1 - n \} $, a contradiction.

Suppose $d_{i0}=1$. Then $ b = B $ which implies $ i = k $ or $ n = 2k + 1 $. If $b=B=0$, then $ r = 0 $ and $ s = k $. But we have \[d_{k0} = q^{k(k + 1)}\qb{n - k - 1}{k} > 1,\]a contradiction. If $ n = 2k + 1 $, then $ b = B = k - i $ which implies $ r = k - i $ and $ s = i $. But we have\[d_{i0} = q^{(k - i)(k + i - 1)/2}\qb{k - i + 1}{1}\qp{k - i} \times q^{i(k + 1)} > 1,\]a contradiction.
\end{proof}

\begin{lem}\label{S_ij}
	
Let $ u(m)\subseteq S $ be an $ m $-subspace. For a fixed vertex $ x \in A(S,u(m),k,i) $, we have $|N(x)\cap A(S,u(m),k,j)|= d_{ij} $, where $0\le i,j\le k$ with $i\not=j$.
	
\end{lem}

\begin{proof}

Let $ x \in A(S,u(m),k,i)$. If there is $y\in A(S,u(m),k,j)$ such that $y\in N(x)$, then we have $\dim(x\cap y)=0$.
We choose $ k $ independent vectors to form an $ k $-dimensional subspace $ y\in A(S,u(m),k,j)\cap N(x) $ according to the following three strategies such that $ y $ admits a direct sum decomposition, that is $ y = y_1 \oplus y_2 \oplus y_3 $.

Firstly, in order that $y\in A(S,u(m),k,j)$, we choose $ j $ independent vectors to form $ y_1 $ such that $ y_1 \subseteq u(m) $, $ y_1 \cap x = \textbf{0} $ and $ \dim y_1 = j $. By Lemma \ref{S_i}, we have	
\[|\overline{Y_1}| =a(u(m),(x\cap u)(i),j,0)= q^{ij}\qb{m - i}{j},\]
where $\overline{Y_1}=\{y_1: y_1 \subseteq u(m), y_1 \cap x = \textbf{0}, \dim y_1 = j\}$.

%%%%%%%%%%%%%%%%%%%
Let $y_1\in \overline{Y_1}$. We now choose $ k - j $ independent vectors not belonging to $ y_1 $ to form $ y_2 $ such that $ y_2\subseteq u(m) + x $, $ y_2\cap u(m) =\textbf{0}$, $ y_2 \cap (y_1 + x) = \textbf{0} $ and $ \dim y_2 =r $. Denote $\overline{Y_2}=\{y_2: y_2 \subseteq u(m) + x, y_2 \cap u(m) = y_2 \cap (y_1 + x) = \textbf{0}, \dim y_2 = r\}$. Then we have
$$\begin{array}{rcl}
|\overline{Y_2}|&=&\prod_{l = 0}^{r - 1}\left(q^{m + k - i} - q^{m + l} - q^{k + j + l} + q^{i + j + 2l}\right) \left(q^{r(r - 1)/2}\qp{r} q^{jr}\right)^{-1}\\
&=&q^{r(i + j + r - 1)}\frac{\qp{m - i - j}\qp{k - i}}{\qp{m - i - j - r}\qp{k - i - r}} \left(q^{r(r - 1)/2}\qp{r} q^{jr}\right)^{-1}\\
&=&q^{r({r - 1} + 2i)/2} \qb{m - i - j}{r}\qb{k - i}{r}\qp{r},
\end{array}$$
where $ 0 \leqslant r \leqslant \min\{m - i - j, k - i\} $.
%For the moment, we consider $ U + X $ as the whole space. Applying the inclusion-exclusion principle and the rule of product, we pick the $ r $ independent vectors one by one, thus the number of the choices of the $ r $ independent vectors is:
	
%\[\prod_{l = 0}^{r - 1}\left(q^{m + k - i} - q^{m + l} - q^{k + j + l} + q^{i + j + 2l}\right)
%\iffalse = \prod_{l = 0}^{r - 1}\left((q^m - q^{i + j + l})(q^{k - i} - q^l)\right) \fi
%= q^{r(i + j + r - 1)}\frac{\qp{m - i - j}\qp{k - i}}{\qp{m - i - j - r}\qp{k - i - r}}\]
	
%\noindent But two choices form the same subspace up to a transformation of $ \GL_r(\BF_q) $, and $ \# \GL_r(\BF_q) = q^{r(r - 1)/2}\qp{r} $. Also, every vector in such a choice can be translated by vectors of $ Y_1 $, (for example: $ \langle y_1, y_2 \rangle = \langle y_1, y_1 + y_2 \rangle $,) thus there are $ q^{jr} $ repetitions. After we get rid of these repetitions, we have:

%\[\# = q^{r({r - 1} + 2i)/2} \qb{m - i - j}{r}\qb{k - i}{r}\qp{r}\]

%\noindent According to the rule of sum, different $ r $ should be counted, thus every $ 0 \leqslant r \leqslant \min\{m - i - j, k - i\} $ should be counted, where the upper bound here is to ensure that the corresponding Gaussian coefficients nonzero.
%%%%%%%%%%%%%%%%%%%%%%%%%%%%%%%%%%%%%%%%%%%%%%
Given $y_1\in \overline{Y_1}$ and $y_2\in \overline{Y_2}$. Let  $ s = k - j - r $. We choose $ s $ independent vectors to form $ y_3 $ such that $ y_3\cap (u((m) + x) = 0 $ and $ \dim y_3 = s $. By Lemma \ref{S_i} and eliminating the repetitions, we have
\[|\overline{Y_3}| =a(S,(u(m)+x)(m+k-i),s,0)q^{-s(j+r)}
= q^{s(s + m - i)}\qb{n - m - k + i}{s},\]
where $\overline{Y_3}=\{y_3: y_3 \cap (u(m)+x) = \textbf{0}, \dim y_3 = s\}$ and $0\le s\le n-m-k+i$ which implies $ m + 2k - n - i - j \leqslant r \leqslant k - j $.	

In conclusion, we have that
$$d_{ij} = |\overline{Y_1}| \times \sum_{r = b}^{B}\left(|\overline{Y_2}| \times |\overline{Y_3}|\right),$$ where $ b = \max\{0, m +2k - n - i - j\}  $ and $ B = \min\{m - i - j, k - i, k - j\} $. Thus the statement follows.	
\end{proof}

\begin{lem}\label{three}
	
Let $ w_1, w_2, w_3 \in V(K_q(2k, k)) $ with $ w_1w_2\in E(K_q(2k, k))$ and $w_{1}w_3, w_{2}w_3 \notin E(K_q(2k, k))$. Then there is $ w_4 \in V(K_q(2k, k)) $ such that $ w_{1}w_4, w_{2}w_4, w_{3}w_4\in E(K_q(2k, k))$.	
	
\end{lem}

\begin{proof} Since $ w_1w_2\in E(K_q(2k, k))$, we have $\dim w_{12}=0$	which implies $ S = w_1 \oplus w_2 $. Denote $ \dim w_{13} = \alpha$ and $\dim w_{23} = \beta $. Then $1\le \alpha,\beta\le k-1$ and $\alpha+\beta\le k$. Let $ \gamma = k - \alpha - \beta $. Then $0\le \gamma\le k-2$. Assume, without loss of generality, that $ w_1 = \langle e_1, \ldots, e_k \rangle$, $ w_2 = \langle f_1, \ldots, f_k \rangle$  and $ w_3 = \langle e_1, \ldots,  e_\alpha \rangle
\oplus \langle f_1, \ldots, f_\beta \rangle \oplus \langle e_{\alpha + 1} + f_{\beta + 1}, \ldots, e_{\alpha + \gamma} + f_{\beta + \gamma} \rangle $ if $\gamma\ge 1$; else  $ w_3 = \langle e_1, \ldots,  e_\alpha \rangle
\oplus \langle f_1, \ldots, f_\beta \rangle $.
 If $\gamma\ge 2$, we choose
\[ w_4 = \langle e_{\alpha + 1} + f_{\beta + 2}, \cdots, e_{\alpha + \gamma - 1} + f_{\beta + \gamma} \rangle \oplus \langle e_{\alpha + \gamma} + f_1, \cdots, e_k + f_{\beta + 1} \rangle \oplus \langle e_1 + f_{\beta + \gamma + 1}, \cdots, e_\alpha + f_k \rangle; \]if $ \gamma = 1 $, we choose
\[ w_4 =  \langle{e_{\alpha + 1} + f_1, \cdots, e_k + f_{\beta + 1}}\rangle \oplus \langle{e_1 + f_{\beta + 2}, \cdots, e_{\alpha} + f_k}\rangle; \]if $ \gamma = 0 $, we choose
\[w_4 = \langle{e_{\alpha + 1} + f_1, \cdots, e_k + f_\beta}\rangle \oplus \langle{e_1 + f_{\beta + 1}, \cdots, e_{\alpha} + f_k}\rangle.\] In each cases, we have $ \dim ((w_4 + w_1) / w_1) = \dim ((w_4 + w_2) / w_2) = k $ which implies $ w_{14} = w_{24} = \textbf{0} $. Since for any $  1 \le i \le k$, $e_i, f_i \in w_3 + w_4 $, we have  $ w_{34} = \textbf{0} $. Thus $ w_{1}w_4, w_{2}w_4, w_{3}w_4\in E(K_q(2k, k))$.	\end{proof}

%For a vertex $ v \in V $ (resp. a subset $ T \subseteq V $), we denote by $ N(v) $ (resp. $ N(T) $) the neighbors of $ v $ (resp. $ T $).

\begin{lem}\label{coclique}
	
Let $ w_1, w_2 \in V(K_q(2k, k)) $ with $ w_1 w_2 \notin E(K_q(2k, k)) $. Then we have $ N(w_1) \cup N(w_2) \subseteq H(\{w_1, w_2\}) $.
\end{lem}

\begin{proof}	
Obviously, we have $ N(w_1) \cap N(w_2) \subseteq H(\{w_1, w_2\}) $. Assume $N(w_1) \setminus N(w_2)\not=\emptyset $ and let $ x \in N(w_1) \setminus N(w_2) $. Then $ x \cap w_1 = \textbf{0} $ but $ x \cap w_2 \ne \textbf{0} $. By Lemma \ref{three}, there exists $y\in  V(K_q(2k, k)) $ such that $ y \in N(w_1) \cap N(w_2) $  and $ xy\in E(K_q(2k, k)) $. So $ x \in H(\{w_1, w_2\}) $. Thus we have $ N(w_1) \cup N(w_2) \subseteq H(\{w_1, w_2\}) $.
\end{proof}

Now we are going to prove Theorem \ref{thm_main}.

\vskip.2cm

{\bf Proof of Theorem \ref{thm_main}} We just need to show that $ h(K_q(n, k)) \le 2 $. We consider the following two cases.

\vskip.2cm
{\bf Case 1} $ n \geqslant 2k + 1 $.

Denote by $ e_1,e_2,\ldots,e_n $ the coordinate vectors. Let $ m = k + 1$ and $ u(k+1) = \langle e_1, \ldots, e_{k + 1} \rangle $. By Lemma \ref{S_i}, we have $ a(S,u(k+1),k,0) \ne 0 $. Let $ w_1 = \langle e_1, \ldots, e_k \rangle $ and $ w_2 = \langle e_2, \ldots, e_{k + 1} \rangle $. Define $ T = \{w_1, w_2\} $. We will show that
$T$ is a $P_3$-hull set of $K_q(n, k)$.

Note that $ w_1,  w_2\in N(w) $ for all   $ w\in A(S,u(k+1),k,0) $. Hence $ A(S,u(k+1),k,0) \subseteq H(T) $. For any $x\in A(S,u(k+1),k,i)$ with $1\le i\le k$, $|N(x)\cap A(S,u(k+1),k,0)|= d_{i0}$ by Lemma \ref{S_ij}. Since $ d_{i0} \geqslant 2 $ by Lemma \ref{d_ij}, we have $ A(S,u(k+1),k,i) \subseteq I[A(S,u(k+1),k,0)] \subseteq H(T) $ for all $ 1\le i\le k $. Thus $H(T)=V(K_q(n, k))$ which implies $ h(K_q(n, k)) = 2 $.

\vskip.2cm
{\bf Case 2} $n=2k$.

%\begin{lem}\label{diameter}
		
%When $ k \geqslant 2 $, the diameter of $ K_q(2k, k) $ is $ 2 $.	
	
%\end{lem}

%\begin{proof}
	
%Let $ W_1, W_2 \in V(K_q(2k, k)) $ with $W_1 W_2 \notin E(K_q(2k, k)) $. Then
%$ \dim W_{12} = a \geqslant 1 $. Without loss of generality, assume $ W_1 = \langle e_1, \ldots, e_a \rangle \oplus \langle e_{a + 1}, \ldots, e_k \rangle$ and $ W_2 = \langle e_1, \ldots, e_a \rangle \oplus \langle e_{k + 1}, \ldots, e_{2k - a} \rangle  $. Then $ W_3 := \langle e_{a + 1} + e_{k + 1}, \ldots, e_k + e_{2k - a} \rangle \oplus \langle e_{2k - a + 1}, \ldots, e_{2k} \rangle $ satisfies $ \dim W_{13} =\dim W_{23} = 0 $ which implies $W_1W_3,W_2W_3\in E(K_q(2k, k)) $. Thus the statement follows.	
%\end{proof}

 Choose $ w_1, w_2 \in V(K_q(2k, k)) $ with $w_1w_2\notin E(K_q(2k,k))$. Define $ T = \{w_1, w_2\} $. We will show that
$T$ is a $P_3$-hull set of $K_q(2k, k)$. By Lemma \ref{coclique}, $ N(w_1) \cup N(w_2) \subseteq H(T) $. Let $w_3\in V(K_q(2k, k))\setminus (N(w_1) \cup N(w_2)) $, say $w_1w_3\notin E(K_q(2k,k))$. Assume $ \dim w_{13} = k - a $. Then $ 1 \leqslant a \leqslant k - 1  $ and we have
$$\begin{array}{rcl}
&&| \left(N(w_1) \cap N(w_3)\right)|\\
&=&\left(\prod_{l = 0}^{a - 1}\left(q^{k + a} - 2q^{k + l} + q^{k - a + 2l}\right)\right) \times \frac{1}{q^{a(a - 1)/2}\qp{a}q^{a(k - a)}} \times a(S, (w_1 + w_2)(k+a), k - a, 0) \\
 &=& q^{(k - 1) a}\times q^{-a(a - 1)/2}\qp{a}  \times q^{(k - a)(k + a)} \times q^{-a(k - a)} \\
 &=& q^{k^2 - \binom{a + 1}{2}} \times \qp{a} \\
 &\ge & 2,
\end{array}$$which implies $w_3\in H(T)$. Thus we have $ h(K_q(2k, k)) = 2 $.\q

\section{Hull number of $J_q(n, k)$}

In this Section, we determine the exact value of   $ h(J_q(n, k))$ for $n\ge 2k$ and  $ k \gs 2 $. 

\vskip.2cm
{\bf Proof of Theorem \ref{thm_main2}} Let $ \{e_1, e_2, \cdots, e_n\} $ be the coordinate vectors of $S$. For short, we denote $ \spann{0} := \mathbf{0} $ the zero space and $ \spann{l} := \spann{e_1, e_2, \cdots, e_l} $ for some positive integer $ l $. Let $v_1,v_2\in V(J_q(n, k))$, where
\[v_1 = \spann{k - 2} \oplus \spann{e_{k - 1}, e_{k}} \qquad \mbox{and} \qquad v_2 = \spann{k - 2} \oplus \spann{e_{k + 1}, e_{k + 2}}.\]
 Then $v_1v_2\notin E(J_q(n,k))$. Let $T=\{v_1,v_2\}$. We will show that
$T$ is a $P_3$-hull set of $K_q(2k, k)$. Let 
$u_1 = \spann{k - 2} \oplus \spann{e_{k - 1}, e_{k + 1}}$, $u_2 = \spann{k - 2} \oplus \spann{e_{k - 1}, e_{k + 2}}$,
$u_3 = \spann{k - 2} \oplus \spann{e_k, e_{k + 1}}$ and $u_4 = \spann{k - 2} \oplus \spann{e_k, e_{k + 2}}$. Then $v_1,v_2\in \cap_{i=1}^4N(u_i)$ which implies 
 $ \{u_1, u_2, u_3, u_4\} \subseteq I[T] $. Denote	
\begin{align*}
C_1 &= \{w \in V(J_q(n, k)): (\spann{k - 2} \oplus \spann{e_{k - 1}}) \subseteq w\}, \\
C_2 &= \{w \in V(J_q(n, k)): (\spann{k - 2} \oplus \spann{e_{k}}) \subseteq w\}.
\end{align*}
 For any $ w \in C_1 \setminus \{u_1, u_2\} $, we have $ \{u_1, u_2\} \subseteq N(w) $. Hence we have $ C_1 \subseteq I[T] $. Similarly, we have $ C_2 \subseteq I[T] $.
For $ 0 \ls i \ls k - 2 $, denote	
\[D_i = \{w \in V(J_q(n, k)): \spann{k - 2 - i} \subseteq w\} \subseteq V(J_q(n, k)).\]Then $D_{k - 2} = V(J_q(n, k))$ and
\[D_0 \subseteq D_1 \subseteq \cdots \subseteq D_{k - 2}.\]
For any $ w = (\spann{k - 2} \oplus \spann{x_1, x_2}) \in D_0 \setminus (C_1 \cup C_2) $, there exists $ w_1 = (\spann{k - 2} \oplus \spann{e_{k - 1}, x_1}) \in C_1 $ and $ w_2 = (\spann{k - 2} \oplus \spann{e_{k}, x_1}) \in C_2 $ such that $ \{w_1, w_2\} \subseteq N(w) $. Hence $ D_0 \subseteq H(T) $.

Now we prove that for any $ 1 \ls i \ls k - 2 $,  $ D_i \subseteq I[D_{i - 1}] $. We just need to show that for any $ w \in D_i $, $ |N(w) \cap D_{i - 1}| \gs 2 $. Let $ w = \spann{k - 2 - i} \oplus \spann{x_1, x_2, \cdots, x_i, x_{i + 1}, x_{i + 2}} \in D_i$. Set
$w' = \spann{k - 2 - (i-1)} \oplus \spann{x_1, \cdots, x_i, x_{i + 1}}$ and $w'' = \spann{k - 2 - (i-1)} \oplus \spann{x_1, \cdots, x_i, x_{i + 2}}$. Then $w',w''\in D_{i - 1}$ and
 $ \{w', w''\} \subseteq N(w) $. Thus $T$ is a $P_3$-hull set of $K_q(2k, k)$ and the statement follows.\q

\subsection*{Acknowledgements}

This research was supported by  the National Natural Science Foundation of China (Grant 12171272 \& 12161141003 \& 11971158).

\iffalse

\begin{rem}
	
A generalized $ q $-Kneser graph $ K_q(n, k, t) $ has the same vertex set as an ordinary $ q $-Kneser graph, but two vertices $ W_1 $ and $ W_2 $ are adjacent if $ \dim W_{12} < t $, thus $ K_q(n, k) $ is a spanning subgraph of $ K_q(n, k, t) $, so the $ P_3 $-hull number of a generalized $ q $-Kneser graph is also $ 2 $.
	
\end{rem}

\fi


\addcontentsline{toc}{chapter}{Bibliography}
\begin{thebibliography}{99}
\bibitem{AMSSW} J. Araujo, G. Morela, L. Sampaio, R. Soares and V. Weber, Hull number: $P_5$-free graphs
and reduction rules, Electronic Notes in Discrete Mathematics 44(2013), 67-73.

%\bibitem{Barbosa}
%R.M. Barbosa,  E.M.M. Coelho, M.C. Dourado, D. Rautenbach and J.L. Szwarcfiter, On the carath\'{e}odory number for the convexity of paths of order three, SIAM J. on Discrete Math., 26(2012), 929-939.

\bibitem{Blokhuis}
A. Blokhuis, A.E. Brouwer and T. Sz\H{o}nyi, On the chromatic number of $q$-Kneser graphs, Des. Codes Cryptogr. 65(2012), 187-197.

\bibitem{Bollobas}
B. Bollob\'{a}s, The Art of Mathematics: Cofee Time in Memphis, Cambridge UniversityPress, 2006.

\bibitem{Bresar}
 B.  Bres\u{a}r  and  M.  Valencia-Pabon,   On  the $P_3$-hull  number  of  Hamming  graphs, Discrete Appl. Math., 282(2020), 48-52.

\bibitem{CLLL2021}
M. Cao, K. Liu, M. Lu and Z. Lv, Treewidth of the $ q $-Kneser graphs, arXiv: 2101.04518v2.

\bibitem{Centeno}
 C.C.  Centeno,  L.D.  Penso,  D.  Rautenbach  and  V.G.  Pereira  de  S\'{a},   Geodeticnumber versus hull number in $P_3$-convexity, SIAM J. on Discrete Math., 27(2)(2013), 717-731.
 
 \bibitem{DG} M.C. Dourado, J.G. Gimbel, J. Kratochv\'{i}l, F. Protti and J. L. Szwarcfiter,
On the computation of the hull number of a graph, Discrete Math., 309(2009), 5668–5674.

%\bibitem{GR2001}
%C. Godsil and G. Royle, Algebraic graph theory, Springer, 2001.

\bibitem{GK2015}
C. Godsil and K. Meagher, Erd\H{o}s-Ko-Rado Theorems: Algebraic Approaches, Cambridge University Press, 2015.

\bibitem{Cohelo}
 E.M.M. Cohelo, M.C. Dourado and R.M. Sampaio, In approximability results for graph convexity parameters, Theoret. Comput. Sci., 600(2015), 49-58.

\bibitem{GPTVV2021}
L.N. Grippo, A. Pastine, P. Torres, M. Valencia-Pabon and J.C. Vera, On the $ P_3 $-hull number of Kneser graphs, Electron. J. Combin, 28(2021), \#P3.32.

\bibitem{Huang}
L. Huang, B. Lv and K. Wang, Erd\H{o}s-Ko-Rado theorem, Grassmann graphs and $p^s$-Kneser graphs for vector spaces over a residue class ring, J. Combin. Theory Ser. A, 164(2019), 125-158.


\bibitem{Kneser}
M. Kneser,  Aufgabe 360, Jahresber. Deutsch. Math.-Verein. 58 (1955), 27.

\bibitem{KK} R. K\"{o}tter and F.R. Kschischang, Coding for errors and erasures in random network coding, IEEE Trans.
Inf. Theory 54(2008), 3579-3591.


\bibitem{Pan} M. Pankov, Geometry of Semilinear Embeddings: Relations to Graphs and Codes, World Scientific,
New Jersey, London, Singapore (2015).

\bibitem{PWL} W. Pullan, X.-W. Wu and Z. Liu, Construction of optimal constant-dimension subspace codes, J. Comb.
Optim. 31(2016), 1709–1719.

\bibitem{Wan} Z.-X. Wan, Geometry of Classical Groups over Finite Fields, 2nd edn. Science Press, Beijing, New
York (2002).


\bibitem{Lovasz}
L. Lov\'{a}sz,  Kneser's conjecture, chromatic number, and homotopy, J. Combin. Theory Ser. A, 25(1978), 319-324.

\bibitem{Lv}
B. Lv and K. Wang, The eigenvalues of $q$-Kneser graphs, Discrete Math, 312(2012), 1144-1147.

\bibitem{Lv2}
B. Lv and K. Wang, The energy of $q$-Kneser graphs and attenuated $q$-Kneser graphs, Discrete Appl. Math, 161(2013), 2079-2083.




\end{thebibliography}
\end{document}